\documentclass[11pt,tbtags]{amsart}
\usepackage{amssymb}
\usepackage{amsmath}
\usepackage{amsthm}
\usepackage[swedish, english]{babel}
\usepackage[latin1]{inputenc}
\usepackage{psfrag}   
\usepackage{graphicx}
\usepackage[square, comma, numbers, sort&compress]{natbib}

\theoremstyle{plain}
\newcommand{\E}{\mathbb E}
\newcommand{\R}{\mathbb R}

\newcommand\supp{\operatorname{supp}}
\newcommand{\ep}{\epsilon}

\newcommand\lrabs[1]{\left|#1\right|}

\newcommand\Bigset[1]{\ensuremath{\Bigl\{#1\Bigr\}}}

\newcommand\bxt{B^x_t}

\newcommand\og{\overline g}
\newcommand\tO{t_0}
\newcommand\ti{t_1}
\newcommand\dc{\gamma}
\newcommand\frpsimu{\Bigl(\frac{\psi_x}{\mu}\Bigr)}

\newenvironment{example}[1][Example]{\begin{trivlist}
\item[\hskip \labelsep {\bf Example}.]}{\end{trivlist}}
\newenvironment{remark}[1][Remark]{\begin{trivlist}
\item[\hskip \labelsep {\bf Remark}.]}{\end{trivlist}}
\newtheorem{theorem}{Theorem}[section]
\newtheorem{lemma}[theorem]{Lemma}
\newtheorem{corollary}[theorem]{Corollary}

\newtheorem{definition}[theorem]{Definition}

\theoremstyle{definition}

\newenvironment{romenumerate}[1][0pt]{
\addtolength{\leftmargini}{#1}\begin{enumerate}
 }{\end{enumerate}}

\begingroup
  \divide\count255 by 60
  \multiply\count255 by -60
  \advance\count255 by \time
  \ifnum \count255 < 10 \xdef\klockan{\the\count1.0\the\count255}
  \else\xdef\klockan{\the\count1.\the\count255}\fi
\endgroup

\title{Feynman--Kac theorems for generalized diffusions}
\author[Erik Ekstr\"om, Svante Janson and Johan Tysk]
{Erik Ekstr\"om$^{1}$, Svante Janson$^2$ and Johan Tysk$^{1}$}
\address{Uppsala University, Box 480, 75106 Uppsala, Sweden.}
\date{October 22, 2012} 
\thanks{$^1$ Supported by the Swedish Research Council (VR)}
\thanks{$^2$ Supported by the Knut and Alice Wallenberg Foundation}

\begin{document}

\begin{abstract} 
We find Feynman--Kac type representation theorems for generalized diffusions. 
To do this we need to establish existence, uniqueness and regularity results for 
equations with measure-valued coefficients.
\end{abstract}

\maketitle

\section{Introduction}\label{S:intro}
Generalized diffusions, also referred to as gap diffusions,  provide a useful extension of the concept of one-dimensional diffusions:
they allow jumps, but only to neighboring elements in the state space, and thus provide a
 unified framework
for processes on discrete spaces as well as for processes on the real line, see \cite{EH}, \cite{EHJT}, \cite{K}, \cite {KW} and \cite{M}. 
In the next
section of the present paper we present generalized diffusions as time changes, 
using so-called speed measures, of Brownian motion, which is the customary way of defining them. We focus on generalized diffusions that
also are local martingales.  

We establish Kolmogorov backward equations for expected values
$U(x,t)$ of functions of generalized diffusions at time $t$ for processes starting at the point $x$.  The backward equation takes the form
$$2m(dx)U_t(x,t)=U_{xx}(x,t)$$ where $m$ denotes the speed measure. This equation is to be interpreted as an equality in the 
distributional sense. 
When the process is a regular diffusion, $m$ is nothing but the multiplicative inverse of the diffusion coefficient and the equation reduces to the usual backward (heat) equation.
We provide related Feynman--Kac theorems connecting solutions to the backward equation to expected values of functions of generalized
diffusions and vice versa. 
We also study regularity of solutions to the backward equation. A consequence of our results is 
that 
(under suitable conditions) 
the measure $U_{xx}$ is absolutely continuous with
respect to the speed measure, which can be thought of as parabolic regularity in our present setting.

The paper is organized as follows. In Section~\ref{S:construction} we introduce the 
family of generalized diffusions under consideration, and we collect some known properties
from existing literature.
In Section~\ref{S:mainresult} our main result, Theorem~\ref{main}, is formulated. It states that 
there is a unique solution satisfying appropriate growth conditions of the backward equation
corresponding to a generalized diffusion. Moreover, this solution coincides with the
stochastic representation, thus establishing a Feynman--Kac type theorem in our setting.
The proof of Theorem~\ref{main} is contained in Sections~\ref{S:uniqueness} 
and \ref{S:existence}. Section~\ref{S:properties} provides conditions under which
certain properties of the initial condition are inherited by the solution, and 
Section~\ref{S:examples} contains some examples. Finally, Section~\ref{S:regularity}
contains a study of regularity of solutions.

\section{Construction of generalized diffusions}
\label{S:construction}
In this section we construct generalized diffusions as
time changes of Brownian motion, and we discuss some of their properties. 
Most of the contents of this section can also be found in \cite{EHJT}, but for the 
convenience of the reader we include them here.

Let $m$ be a nonnegative Borel measure on $\R$. Note that $m$ is allowed to
be (locally) infinite. 
We exclude the trivial case when $m=0$. 
Let $B$ be a Brownian motion with $B_0=0$ and let
$L_u^y$ be its  local time at the point $y$ up to time $u$. 
Then $\bxt:=x+B_t$ is a Brownian motion starting at $x$, 
and its local time at $y$ is $L^{y-x}_u$.
For a given starting point $x\in\R$ we define the increasing process
\begin{equation}\label{gamma}
\Gamma^{x}_u:=\int_\R L_u^{y-x} m(dy).
\end{equation}
We note that $\Gamma^{x}_u \in [0,\infty]$, 
and we define its right-continuous inverse
\begin{equation}\label{axt}
A^{x}_t:=\inf\{u: \Gamma_u^x>t\}.
\end{equation}
Since $m$ is non-zero, $\Gamma^x_u\to\infty$ as $u\to\infty$, so
$A_t^x<\infty$.  
The process 
\begin{equation}
X^{x}_t:= \bxt=x + B_{A^x_t}
\end{equation}
will be called a generalized diffusion with speed measure $m$ and starting point $x$.

Although $B^x_{0}=x$,
$A^x_0$ may be strictly positive and thus in general 
$X^x_0\neq x$. Indeed, this is the case when $x$ does not belong to the support of $m$, see 
\cite{EHJT}, Lemma 3.3. We define $A^x_{0-}=0$ and $X_{0-}^x=x+B_{A^x_{0-}} = x$  and thus we 
allow the possibility that $X_{0-}^x\neq X^x_0$. 

We list some further properties of generalized diffusions. 
\begin{romenumerate}
\item\label{localmart}
$X$ is a local martingale for any initial point $x$ if and only if 
$\supp m\cap(-\infty,-b]\not= \emptyset$ and $\supp m\cap[b,\infty)\not= \emptyset$ for all $b>0$,
see Theorem~7.3 in \cite{EHJT}.
\item\label{mart}
$X$ is a martingale for any initial point $x$ if and only if the speed measure $m$ satisfies
\begin{equation}
\label{martingality}
\int_{-\infty}^{x} \vert y\vert \,m(dy)=\int_{x}^{\infty} \vert y\vert\,m(dy)=\infty
\end{equation}
for any $x\in\R$, see Theorem~7.9 in \cite{EHJT}.
\item \label{vaguelim} 
If $m_n$ is a sequence of speed measures that converges vaguely to $m$ in
the sense that
\begin{equation}\label{vague}
  \int \phi\,dm_n\to \int \phi\,dm,
\qquad \phi\in C_c^+(\R),
\end{equation}
then the corresponding 
$A^{x,n}_t\to A^x_t$ and $X^{x,n}_t\to X^x_t$ a.s.\ as $n\to\infty$, for
every $x$ and $t>0$,
see Lemma 3.12 in \cite{EHJT}.
\item\label{diffusion}
If $dm(y)=\frac{dy}{\sigma^2(y)}$ for some continuous non-vanishing function
$\sigma$, then $X^x_t$ is a weak solution of
$$dX^x_t=\sigma(X^x_t)\,dW_t,$$
$X_0=x$ (with $W$ being Brownian motion), in which case $X^x$ is a diffusion. 
Note that the speed measure $m$ measures the 
inverse of speed rather than speed.
\end{romenumerate}

\section{A Feynman--Kac type theorem for generalized diffusions}
\label{S:mainresult}

Throughout the rest of this article we assume that $m$ is a locally finite and
nonnegative Borel measure on $\R$ such that  
\begin{equation}
\label{unbounded_support}
\left\{\begin{array}{l}
\supp m\cap(-\infty,-b]\not= \emptyset\\
\supp m\cap[b,\infty)\not= \emptyset\end{array}\right.\mbox{ for all }b>0.
\end{equation}
The corresponding generalized diffusion is then a local martingale 
by \ref{localmart} above. Let 
\begin{equation}
\label{phi}
\Phi(x):=\left\{\begin{array}{rl}
2\int_{[0,x)} y\,m(dy) & x\geq 0\\
2\int_{[x,0)}  y \,m(dy) & x<0\end{array}\right.
\end{equation}
and
\begin{equation}
\label{Psi}
\Psi(x):=
\int_0^x\Phi(y)\,dy .
\end{equation}
Then 
$\Phi$ is non-decreasing with $\Phi(0)=0$, and
$\Psi$ is a non-negative convex function with left 
derivative $\Phi$ and second derivative 
$\Psi^{''}(x) = 2\vert x\vert m(dx)$ (in the distribution sense).
Note that it follows from \eqref{unbounded_support} that $\Psi$ grows at least linearly as $\vert x\vert\to\infty$.
Hence there is a constant $C_1$ such that, for all real $x$,
\begin{equation}\label{xpsi}
  |x|\le C_1(\Psi(x)+1).
\end{equation}

\begin{definition}\label{Dog}
Let $g$ be a continuous function. We define $\overline g$ to be the function
that  
agrees with $g$ on the support of $m$, and which is affine outside the support.
\end{definition}

We now state our main result.

\begin{theorem}
\label{main}
Suppose that $m$ is a locally finite nonnegative Borel measure on $\R$ such
that  \eqref{unbounded_support} holds.
Let $g$ be a continuous function such that $g(x)= o(\Psi( x))$ as $\vert x\vert\to\infty$.
Then there exists a unique continuous function $U:\R\times[0,\infty)\to \R$
  such that $U(x,t)= o(\Psi( x))$ as $\vert x\vert\to\infty$ 
locally uniformly in $t$ and such that 
\begin{equation}
\label{mpde1}
2m(x)U_t(x,t)=U_{xx}(x,t)
\qquad\text{for }t>0
\end{equation}
holds in the sense of distributions with the initial values 
\begin{equation}\label{mpde2}
  U(x,0)=\overline g(x).
\end{equation}

Moreover, the function $U$ is given by a stochastic representation 
\begin{equation}
  \label{mainu}
U(x,t)=\E g(X^x_t)=\E\overline g(X^x_t),
\end{equation} 
where $X^x_t$ is the generalized diffusion with 
speed measure $m$ and starting point $x$.
\end{theorem}

\begin{remark}
The distribution $mU_t$ appearing in \eqref{mpde1} is to be interpreted as
the distribution derivative $(Um)_t$. 
(We regard here $m=m(x)$ as a distribution in $\R\times(0,\infty)$, 
independent of $t$.) 
Hence the partial differential equation 
\eqref{mpde1} is equivalent to 
\begin{equation}
\label{distribution}
-2\iint U\varphi_t\,m(dx)\,dt = \iint U\varphi_{xx}\,dx\,dt
\end{equation} 
for all $\varphi\in \mathcal D=\mathcal D(\R\times(0,\infty))$, the space of
smooth functions of compact support.

We shall see in Theorem \ref{regularity} that the solution $U$ actually has 
a continuous derivative $U_t$ in $\R\times (0,\infty)$; for such
functions, $m U_t$ in 
\eqref{mpde1} can, equivalently, also be interpreted in the usual sense as
the product of the distribution $m(x)$ in $\R\times(0,\infty)$ and the
continuous function $U_t$. We may still have to interpret $U_{xx}$ in
distribution sense, but it is equivalent to interpret \eqref{mpde1} as an
equation of distributions on $\R$, for every fixed $t>0$.
\end{remark}

\begin{remark} We can think of $g(x)$ as $u(x,0-)$. Thus from time $0-$ to time $0$, $u(x,t)$
changes from $g$ to $\overline{g}$. 
\end{remark}

\section{Uniqueness of solutions}
\label{S:uniqueness}

In this section we prove uniqueness of solutions to
\eqref{mpde1}--\eqref{mpde2}. 
We begin with a maximum principle on bounded domains.

For a fixed $T>0$ and $A>0$, let $D=D_A:=[-A,A]\times[0,T]$. Denote by $D^\circ=(-A,A)\times(0,T)$ the interior of $D$,
and let $\partial D=[-A,A]\times\{0\}\cup\{-A,A\}\times[0,T]$ denote its parabolic boundary.

\begin{lemma}
\label{bounded}
Suppose that $f\in C(D)$ and $f_t\in C(D^\circ)$. If $2f_t m \geq f_{xx}$ on $D^\circ$ (as distributions) and 
$f\geq 0$ on $\partial D$, then $f\geq 0$ on $D$.
\end{lemma}

\begin{proof}
Let $\ep>0$, and define 
\begin{equation}
\label{F}
F(x,t):=f(x,t) +\ep (1+t+A^2-x^2).
\end{equation}
Then $F\in C(D)$ and $F(x,t)\geq \ep$ on $\partial D$.
Let $E:=\{(x,t)\in D: F(x,t)\leq 0\}$. Then $E$ is compact.
Let 
\[t_0:=\min\{t\geq 0:(x,t)\in E\mbox{ for some }x\in[-A,A]\},\]
and suppose $t_0<T$. Since $E\cap\partial D=\emptyset$ we have $t_0>0$.
Take $x_0\in[-A,A]$ with $(x_0,t_0)\in E$, i.e.\ $F(x_0,t_0)\leq 0$. Then $-A<x_0<A$, so $(x_0,t_0)\in D^\circ$.

Note that $F_t=f_t+\ep\in C(D^\circ)$. If $t<t_0$, then 
$F(x_0,t)\geq 0\geq F(x_0,t_0)$, so $F_t(x_0,t_0)\leq 0$. Consequently, $f_t(x_0,t_0)\leq -\ep$.
By continuity of $f_t$, there exists a neighborhood $U\subset D^\circ$ of $(x_0,t_0)$ such that
$f_t<0$ in $U$. Thus $f_tm\leq 0$ in $\mathcal D^{\prime}(U)$, and 
$f_{xx}\leq 2f_t m\leq 0$ in $\mathcal D^{\prime}(U)$.

Let $\psi\in C^\infty_0(\R^2)$ with support in the unit ball $B$ such that $\psi\geq 0$ and $\int\psi=1$, 
and define $\psi_\eta(x)=\eta^{-2}\psi(x/\eta)$. Let $V$ be a smaller neighborhood of $(x_0,t_0)$, 
i.e.\ $\overline V\subset U$. Then, for $\eta$ so small that $V+\eta B\subset U$, say $\eta<\eta_0$, we have that
\[(\psi_\eta\mathord{*}f)_{xx}=\psi_\eta\mathord{*} f_{xx}\leq 0\]
in $\mathcal D^{\prime}(V)$. But $\psi_\eta\mathord{*}f\in C^{\infty}(V)$, so this means that 
$(\psi_\eta\mathord{*}f)_{xx}\leq 0$ as a continuous function, i.e.\ pointwise in $V$.

Choose $h$ so small that $[x_0-h,x_0+h]\times\{t_0\}\subset V$. Then, for any $\eta<\eta_0$, $\psi_\eta\mathord{*} f(x,t_0)$
is concave on $[x_0-h,x_0+h]$, and thus 
\[\psi_\eta\mathord{*} f(x_0-h,t_0) + \psi_\eta\mathord{*} f(x_0+h,t_0) \leq 2\psi_\eta\mathord{*} f(x_0,t_0).\]
Letting $\eta\to 0$ we have, since $f$ is continuous,
$\psi_\eta\mathord{*} f (x,t)\to f(x,t)$ for every $(x,t)\in V$, and thus
\[f(x_0-h,t_0) + f(x_0+h,t_0)\leq 2f(x_0,t_0).\]
Consequently, using the definition \eqref{F},
\begin{align*}
 F(x_0&-h,t_0) + F(x_0+h,t_0)-2F(x_0,t_0) \\
&=f(x_0-h,t_0) + f(x_0+h,t_0)-2f(x_0,t_0) 
\\&\hskip10em
-\ep((x_0-h)^2+(x_0+h)^2-2x_0^2)
\\ &\leq -2\ep h^2<0.
\end{align*}
However, by definition $F(x_0,t_0)\leq 0$ while $F(x,t)>0$ for $x\in [-A,A]$ and $0\leq t<t_0$, and thus 
by continuity we have $F(x,t_0)\geq 0$. Hence $F(x_0-h,t_0) + F(x_0+h,t_0)-2F(x_0,t_0)\geq 0$. This is a contradiction, 
which shows that $t_0<T$ is impossible. Consequently, either $t_0=T$ or $E=\emptyset$. In both 
cases, $F(x,t)>0$ for $(x,t)\in[-A,A]\times[0,T)$. By continuity $F\geq 0$ in $D$. Hence, by \eqref{F},
\[f(x,t)+\ep(1+t+A^2-x^2)\geq 0\]
in $D$. Letting $\ep\to 0$ finishes the proof.
\end{proof}

We next extend the maximum principle to an unbounded domain.
Let $D_\infty:=\R\times[0,T]$ and $D_\infty^\circ:=\R\times(0,T)$.

\begin{lemma}
\label{unbounded}
Suppose that $f\in C(D_\infty)$ and $f_t\in C(D_\infty^\circ)$. Also assume that 
$2f_t m \geq f_{xx}$ on $D_\infty^\circ$ (as distributions),
$f\geq 0$ on $\R\times\{ 0\}$, and that $f(x,t)=o(\Psi(x))$ as $\vert x\vert \to\infty$ uniformly in $t\in[0,T]$.
Then $f\geq 0$ on $D$.
\end{lemma}

\begin{proof}
Let $\ep>0$ and define, with $C_1$ as in \eqref{xpsi}, 
\begin{equation*}
\label{h}
h(x,t):=f(x,t) + \ep (1+\Psi(x))e^{C_1t}.
\end{equation*}
By the assumption $f(x,t)=o(\Psi(x))$, 
if $A$ is large enough, then  $h(x,t)\geq 0$ for $\vert x\vert\geq
A$ and $t\in[0,T]$. 
Fix one such $A$. Then 
$h\geq 0$ on $\partial D_A$. Moreover, using \eqref{xpsi}, 
\begin{align*}
h_{xx} &=  f_{xx} +\ep \Psi_{xx}e^{C_1t}= f_{xx}+\ep 2|x|m e^{C_1t}\\
&\leq 2f_t m +\ep 2C_1(1+\Psi(x))me^{C_1t} = 2 h_tm
\end{align*}
on $D^\circ_A$. Hence Lemma~\ref{bounded} applies to $h$ and yields $h\geq 0$ on $D_A$. Since we can choose 
$A$ arbitrarily large, $h\geq 0$ on $D_\infty$. Now, letting $\ep\to 0$ yields $f(x,t)\geq 0$ for $(x,t)\in D_\infty$.
\end{proof}

\begin{lemma}
\label{unbounded2}
Suppose that $f\in C(D_\infty)$. Also assume that 
$2f_t m = f_{xx}$ on $D_\infty^\circ$,
$f= 0$ on $\R\times\{ 0\}$, and that $f(x,t)=o(\Psi(x))$ as $\vert x\vert \to\infty$ uniformly in $t\in[0,T]$.
Then $f=0$ on $D_\infty$.
\end{lemma}

\begin{proof}
Define $F(x,t)=\int_0^t f(x,s)\,ds$. Then $F\in C(D_\infty)$ and $F_t(x,t)=f(x,t)\in C(D_\infty)$.
We have 
\[F_{xxt}=f_{xx}=2f_tm=2(fm)_t=2 (F_tm)_t.\]
Let $G:=F_{xx}-2F_tm\in\mathcal D^{\prime}(D^\circ_\infty)$. Then $G_t=0$, so $G(x,t)=h(x)$ for some
distribution $h\in\mathcal D^\prime (\R)$. 

Fix $\phi\in\mathcal D(\R)$, and let $\psi\in \mathcal D(0,\ep)$ with
$\psi\geq 0$ and $\int \psi=1$. 
Then, with $C$ depending on $\phi$ only,
\begin{align*}
\vert \langle F_{xx}(x,t), \phi(x)\psi(t)\rangle\vert 
&= \lrabs{\iint F(x,t)\phi_{xx}(x)\psi(t) \,dx\,dt}\\
&\leq C\sup_{0\leq t\leq\ep,x\in\supp \phi}\vert F(x,t)\vert\\
&\leq C \ep \sup_{0\leq t\leq\ep,x\in\supp \phi}\vert f(x,t)\vert = o(\ep)
\end{align*}
as $\ep\to 0$. Furthermore, 
\begin{align*}
\vert \langle F_{t}m(x,t), \phi(x)\psi(t)\rangle \vert 
&= \lrabs{\iint f(x,t)\phi(x)\psi(t) \,m(dx)\,dt}\\
&\leq  C \sup_{0\leq t\leq\ep,x\in\supp \phi}\vert f(x,t)\vert = o(1)
\end{align*}
as $\ep\to 0$. Hence $\vert \langle  G,\phi(x)\psi(t)\rangle \vert = o(1)$ as $\ep\to 0$. But 
\[\langle G,\phi(x)\psi(t)\rangle  = \langle h(x),\phi(x)\psi(t)\rangle =\langle h,\phi\rangle \int\psi(t)\,dt = \langle h,\phi\rangle .\]
Hence, letting $\ep\to 0$, $\langle h,\phi\rangle =0$. Since $\phi\in\mathcal D(\R)$ is arbitrary, 
$h=0$ and thus $G=0$. Consequently, $F_{xx}=2F_tm$ on $D^\circ_\infty$.
Moreover, $F(x,0)=0$ and 
\[\vert F(x,t)\vert\leq T\sup_{0\leq t\leq T} \vert f(x,t)\vert = o(\psi(x))\]
as $\vert x\vert \to \infty$. Hence Lemma~\ref{unbounded} applies to $F$ and shows that
$F\geq 0$ on $D_\infty$. Moreover, Lemma~\ref{unbounded} also applies to $-F$ and yields
$-F\geq 0$ on $D_\infty$. Hence $F=0$ on $D_\infty$, which implies that $f=0$ on $D_\infty$.
\end{proof}

\begin{proof}[Proof of uniqueness in Theorem~\ref{main}.]
Assume that $U_1$ and $U_2$ both solve \eqref{mpde1}--\eqref{mpde2}
with the same initial condition $\overline g$, and
that $U_1=o(\Psi(x))$ and $U_2=o(\Psi(x))$ as $\vert x\vert\to\infty$ locally uniformly in $t$. 
Applying Lemma~\ref{unbounded2} to the function $f:=U_1-U_2$ shows that $f=0$ on $D_\infty$, 
so $U_1=U_2$.   
\end{proof}

\section{Existence of solutions}
\label{S:existence}

In this section we prove the existence claim of Theorem~\ref{main}. Indeed, 
we show that the function $U$ given by stochastic representation in
\eqref{mainu} 
solves \eqref{mpde1}--\eqref{mpde2} 
and is $o(\Psi(x))$ as $\vert x\vert\to\infty$.

Throughout this section we assume that $g$ is continuous on $\R$ with 
$g(x)=o(\Psi(x))$ as $\vert x\vert\to\infty$. (Sometimes 
this assumption is further strengtened by assuming that $g$ and some of its derivatives are bounded.)

\begin{lemma}\label{Lpsixbound}
There exists a constant $C_1$ such that
for any x and $t\ge0$,
  \[\E\Psi(X^x_t)\leq (\Psi(x ) +1)e^{C_1t}.\]
\end{lemma}

\begin{remark}
The constant $C_1$ can be chosen as the constant appearing in \eqref{xpsi}.
\end{remark}

\begin{proof}
As in the proof of Theorem~7.9 in \cite{EHJT}, the It\^o--Tanaka formula may
be employed  
to show that, for any $r\ge 0$, 
\begin{equation}
\label{7.12}
\E\Psi(B^x_{A_r^x\wedge H})
\leq \Psi(x ) + \int_0^r \E\vert B^x_{A_t^x\wedge H}\vert\,dt
\end{equation}
for any exit time $H:=\inf \{t:B^x_t\not\in (-a,a)\}$ with $|x|\le a$.
Inserting \eqref{xpsi} yields
\[\E\vert B^x_{A_r^x\wedge H}\vert 
\le C_1\E\Psi(B^x_{A_r^x\wedge H})+C_1
\leq C_1\Psi(x)+C_1 +C_1\int_0^r \E\vert B^x_{A_t^x\wedge H}\vert\,dt,\]
so Gronwall's lemma \cite[Appendix 1]{RY} yields
\begin{equation*}
\E\vert B^x_{A_t^x\wedge H}\vert \leq (C_1\Psi(x)+C_1)e^{C_1t}.
\end{equation*}
Inserting this into \eqref{7.12} gives 
\[\E\Psi(B^x_{A_r^x\wedge H})\leq (\Psi(x ) +1)e^{C_1r}.\]
By Fatou's lemma, letting $a\to\infty$ and thus $H\to\infty$,
\[\E\Psi(X^x_r)=\E\Psi(B^x_{A_r^x})\leq (\Psi(x ) +1)e^{C_1r}.
\qedhere
\]
\end{proof}

\begin{lemma}
\label{loc}
If $K\subset\R$ is a compact set and $T>0$, then 
the set of random variables $\{g(X_t^x):(x,t)\in K\times[0,T]\}$ is
uniformly integrable. 
\end{lemma}

\begin{proof}
By Lemma \ref{Lpsixbound}, $\E\Psi(X^x_t)\le C$  for
$(x,t)\in K\times[0,T]$ 
and some $C<\infty$.
Since $g(x)=o(\Psi(x))$ as $\vert x\vert\to\infty$, 
it follows that the
set of random variables $\{g(X_t^x)\}$ with $(x,t)\in K\times[0,T]$ is
uniformly integrable, see e.g.\ \cite[Theorem~5.4.3 and its proof]{Gut}.
\end{proof}


To prove the existence part of Theorem \ref{main} we consider the function
$U(x,t):=\E g(X^x_t)$. Note that $g(X_t^x)=\overline g(X_t^x)$ a.s. since $X_t^x\in \supp m$ (see \cite[Lemma 3.1]{EHJT}),
so $U(x,t)=\E g(X^x_t)=\E \overline g(X^x_t)$.

\begin{lemma}
\label{cont}
The function $U(x,t)=\E g(X^x_t)$ is continuous and 
$U(x,0)=\og(x)$. Furthermore, $U(x,t)=o(\Psi(x))$ as $\vert x\vert\to\infty$
locally uniformly in $t\ge0$.
\end{lemma}

\begin{proof}
Consider a sequence of points $(x_n,t_n)\in\R\times[0,\infty)$ such that
  $(x_n,t_n)\to(x,t)$ as $n\to\infty$. We may assume, for notational
  simplicity, that $x=0$.
By \eqref{gamma},
\begin{equation}
\Gamma^{x_n}_u=\int_\R L_u^{y}\, m(x_n+dy)=\int_\R L_u^{y}\, m_n(dy),
\end{equation}
where $m_n$ is the translated measure defined by $m_n(S):=m(S+x_n)$ for Borel sets
$S\subseteq\R$. 
If we further define $\nu_n=t_n^{-1}m_n$,
then, by \eqref{axt},
\begin{equation}\label{rom}
  A^{x_n}_{t_n}=\inf\Bigset{u: \int_\R L_u^{y} \,m_n(dy)> t_n}
=\inf\Bigset{u: \int_\R L_u^{y} \,\nu_n(dy)> 1}.
\end{equation}
(Note that this holds also in the case $t_n=0$, when the measure $\nu_n$ only
takes the values 0 and $\infty$.)

Since $m$ is locally finite, $m_n\to m$ vaguely, see \eqref{vague}, and thus
$\nu_n\to\nu:=t^{-1}m$ vaguely. Hence,  \eqref{rom} and its analogue for
$(x,t)$ imply by \ref{vaguelim} in
Section \ref{S:construction} that
$A^{x_n}_{t_n}\to A^x_t$ a.s.
Hence, a.s.,
\begin{equation*}
  X^{x_n}_{t_n}=x_n+B_{A^{x_n}_{t_n}}
\to x+B_{A^x_t} = B_{A^x_t} = X^x_t,
\end{equation*}
and thus 
$ g(X^{x_n}_{t_n})\to g(X^{x}_t)$.
Taking the expectations we obtain, using Lemma \ref{loc},
that
\begin{equation*}
U(x_n,t_n)=\E g(X^{x_n}_{t_n})\to \E g(X^{x}_t) = U(x,t),
\end{equation*}
which shows the continuity of $U(x,t)$.

For $t=0$ we have by \cite[Lemma 3.3]{EHJT} that if $x\in\supp m$, then
$A_0^x=0$ and thus $X^x_0=x$ a.s., so $U(x,0)=g(x)$, while if
$x\notin\supp m$, then $A_0^x$ is a.s.\ the first time $B^x_t$ hits $\supp
m$; hence, if $x\in(a,b)$ where $(a,b)$ is a component of the complement of
$\supp m$, then  $U(x,0)=\frac{x-a}{b-a}g(a)+\frac{b-x}{b-a}g(b)=\og(x)$.
Thus $U(x,0)=\og(x)$ in both cases.

For the final claim we note that for any $\ep>0$ there exists $C_\ep$ such
that 
\begin{equation*}
  |g(x)| \le \ep\Psi(x)+C_\ep,
\end{equation*}
and then by Lemma \ref{Lpsixbound}, for $0\le t\le T$,
\begin{equation*}
 \E |g(X^x_t)| \le \ep\E\Psi(X^x_t)+C_\ep
\le\ep e^{C_1T}\Psi(x)+\ep e^{C_1T}+C_\ep;
\end{equation*}
since $\ep$ is arbitrary, 
this implies that $ \E |g(X^x_t)|/\Psi(x)\to0$ uniformly for $0\le t\le T$
as $|x|\to\infty$ and thus $\Psi(x)\to\infty$.
\end{proof}

In view of Lemma~\ref{cont}, it merely remains to prove that $U$ satisfies
$2m U_t=U_{xx}$ in the sense of distributions. 
(Recall that this means \eqref{distribution}; we will use this form of the
equation below, usually without comment.) 
This is done below by a series of approximations.

\begin{lemma}
\label{pointwise}
Assume that $g$ is bounded and that $m_n$ is a sequence of speed measures
converging vaguely to $m$, see \eqref{vague},
and let $X^{x,n}$ and $X^x$ be the corresponding generalized diffusions. Then
$U_n(x,t):=\E g(X^{x,n}_t)\to \E g(X^x_t) =: U(x,t)$ as $n\to\infty$,
for any $x$ and $t>0$. 
\end{lemma}

\begin{proof}
By \ref{vaguelim} in Section \ref{S:construction}, $X^{x,n}_t \to X^x_t$ almost
surely as $n\to\infty$. 
The result then follows by the continuity of $g$ and bounded convergence.
\end{proof}

\begin{lemma}
\label{lowerbound}
Assume that $m(dx)\geq \ep \,dx$ for some $\ep>0$, and that $g$, $g^\prime$
and $g^{\prime\prime}$ are bounded. 
Then $U(x,t)$ satisfies \eqref{mpde1}.
\end{lemma}

\begin{proof}
First note that if $m$ has a density which is regular enough (for the sake of simplicity, say $C^1$ with a bounded derivative), and
bounded away from 0, then $X$ is the weak solution of a stochastic differential equation, 
and by the standard Feynman--Kac theorem (see for example 
\cite[Theorem 6.5.3]{F}),  
$U$ is the unique bounded classical solution of the initial value problem
\eqref{mpde1}--\eqref{mpde2}. 
In particular, see \eqref{distribution},
\begin{equation}
\label{approximation}
-2\iint U\varphi_t\,m(dx)\,dt = \iint U\varphi_{xx}\,dx\,dt
\end{equation}
for all $\varphi\in \mathcal D$.

Now let $m$ be as specified in the lemma, i.e.\ $m(dx)\geq \ep \,dx$ for some $\ep>0$, and let 
$m_n$ be a sequence of measures with regular densities such that $m_n(dx)\geq  \ep \,dx$ for all $n$
and such that $m_n$ converges to $m$ vaguely. (Such a sequence can be 
constructed as convolutions $m_n:=\psi_n*m$ with a suitable sequence of
regularising kernels  $\psi_n\in\mathcal D(\R)$ in the usual way.)
Denote by $X^{x,n}_t$ the corresponding generalized diffusion, and let $U_n(x,t)=\E g(X^{x,n}_t)$.
Since $m_n$ has a regular density, say $dm_n(y)=\frac{dy}{\sigma^2_n(y)}$ 
where $\sigma_n$ is $C^1$ and bounded, 
$X^{x,n}_t$ is a weak solution of the stochastic differential equation
\begin{equation}
\label{sde}
dY_t=\sigma_n(Y_t)\,dW_t
\end{equation}
with $Y_0=x$ (see \ref{diffusion} in Section \ref{S:construction}). Now, let $Y^{x,n}_t$ denote the {\em strong} solution of \eqref{sde} for some given Brownian motion $W$
(a unique strong solution exists since $\sigma$ is $C^1$). Then, by weak uniqueness, 
$Y^{x,n}_t$ and $X^{x,n}_t$ coincide in law, so $U_n(x,t)=\E g(X^{n,x}_t)=\E g(Y^{n,x}_t)$.
Furthermore, since $\sigma_n$ is bounded, $Y^{x,n}_t$
is a martingale, and by a
comparison result for one-dimensional diffusions (see 
\cite[Theorem~IX.3.7]{RY}) we have  
$Y^{x,n}_t\leq Y^{y,n}_t$ if $x<y$. Consequently, if $x<y$, then
\begin{align*}
\vert U_n(y,t)-U_n(x,t)\vert 
&\leq \E\vert g(Y^{n,y}_t)-g(Y^{n,x}_t)\vert 
\leq C \E\vert Y^{n,y}_t-Y^{n,x}_t\vert\\
&= C\E ( Y^{n,y}_t-Y^{n,x}_t)=C(y-x),
\end{align*}
where $C$ is a Lipschitz constant of $g$. Consequently, $U_n$ is Lipschitz continuous in $x$ uniformly in $n$.

Let $D$ be a global bound for $\vert g^{\prime\prime}\vert$.
We claim that 
\begin{equation}
\label{claim}
\vert U_n(x,t)-g(x)\vert \leq \frac{D}{2\ep}t.
\end{equation}
To see this, consider the function
\[f(x,t)=U_n(x,t)-g(x)+\frac{D}{2\ep}t.\]
Then $f$ is a supersolution, i.e.\ it satisfies 
\[\left\{\begin{array}{ll}
2m_nf_t =2m_n(U_n)_t+(D/\ep)m_n\geq (U_n)_{xx}+D\geq  f_{xx}, & t>0,\\
f(x,0)=U_n(x,0)-g(x)= 0,\end{array}\right.\]
so Lemma~\ref{unbounded} yields $f\geq 0$. Consequently, $U_n(x,t)\geq g(x)-\frac{D}{2\ep}t$.
Similarly, the function $U_n(x,t)-g(x)-\frac{D}{2\ep}t$ is a subsolution, so
$U_n(x,t)\leq g(x)+ \frac{D}{2\ep}t$, which finishes the proof of \eqref{claim}.

Next, using the Markov property and \eqref{claim} we find that 
\begin{align*}
\vert U_n(x,t+h)-U_n(x,t)\vert 
&=\vert \E\left[ g(X^{x,n}_{t+h})-g(X^{x,n}_t)\right]\vert\\
&=\vert \E\left[ U_n(X^{x,n}_h,t) -g(X^{x,n}_t)\right]\vert \leq \frac{D}{2\ep}h
\end{align*}
for $h>0$. It follows from this, together with the uniform Lipschitz continuity in $x$ proven above, that 
$(x,t)\mapsto U_n(x,t)$ is Lipschitz continuous uniformly in $n$.
Consequently, the convergence $U_n(x,t)\to U(x,t)$ guaranteed by
Lemma~\ref{pointwise} is uniform on any compact subset of 
$\R\times(0,\infty)$. 

As noted in \eqref{approximation},
\[0 = 2\iint U_n\varphi_t\,m_n(dx)dt + \iint U_n\varphi_{xx}\,dx\,dt\]
for $\varphi\in\mathcal D$.
By bounded convergence, 
\[ \iint U_n\varphi_{xx}\,dx\,dt\to   \iint U \varphi_{xx}\,dx\,dt\]
as $n\to\infty$.
Moreover, 
\begin{align*}
 \iint U_n\varphi_t\,m_n(dx)dt
 &= \iint (U_n-U)\varphi_t\,m_n(dx)dt + 
 \iint U\varphi_t\, m_n(dx)dt \\
&\to   \iint U\varphi_t\,m(dx)dt
\end{align*}
as $n\to\infty$ since $U_n\to U$ uniformly on $\supp\varphi$ and $m_n\to m$
vaguely. Consequently,  
\[0 = 2\iint U \varphi_t\,m (dx)dt + \iint U\varphi_{xx}\,dx\,dt,\]
for any $\varphi\in\mathcal D$, 
so $U$ is a solution of \eqref{distribution} and thus \eqref{mpde1}. 
\end{proof}

\begin{lemma}
\label{almost}
Assume that $g$, $g^\prime$ and $g^{\prime\prime}$ are bounded.
Then $U(x,t)$ satisfies \eqref{mpde1}.
\end{lemma}

\begin{proof}
For a given speed measure $m$, let $m_n(dx)= m(dx) + n^{-1}\,dx$ and let $U$ and $U_n$ be the 
corresponding stochastic representations. By Lemma~\ref{pointwise}, $U_n\to U$ pointwise 
on $\R\times(0,\infty)$ as $n\to\infty$.
 By Lemma~\ref{lowerbound}, 
\[0=2 \iint U_n \varphi_t\,m_n (dx)dt + \iint U_n\varphi_{xx}\,dx\,dt\]
for any $\varphi\in \mathcal D$. Here
\begin{align*}
\iint U_n \varphi_t\,m_n (dx)dt &=
\iint U_n \varphi_t\,m (dx)dt + \frac{1}{n}\iint U_n \varphi_t\, dx\,dt\\
&\to \iint U \varphi_t\,m (dx)dt
\end{align*}
as $n\to\infty$ by bounded convergence and the fact that the functions $U_n$
are uniformly bounded (by $\sup|g|$).
Similarly, 
\[ \iint U_n\varphi_{xx}\,dx\,dt\to  \iint U\varphi_{xx}\,dx\,dt\]
as $n\to\infty$. It follows that
\[0= 2\iint U \varphi_t\,m (dx)dt + \iint U\varphi_{xx}\,dx\,dt\]
for any $\varphi\in \mathcal D$, which finishes the proof.
\end{proof}

\begin{lemma}
\label{Lbdd}
Assume that $g$ is bounded.
Then $U(x,t)$ satisfies \eqref{mpde1}.
\end{lemma}

\begin{proof}
Let  $g_n:=\psi_n*g$ where $\psi_n$ is a  sequence of
regularising kernels in $\mathcal D(\R)$. Then each $g_n$ satisfies the
conditions of Lemma~\ref{almost}, and thus the corresponding
$U_n(x,t):=\E g_n(X^x_t)$  satisfies \eqref{mpde1}, i.e.
\begin{equation}
\label{approxg0}
0= 2\iint U_n \varphi_t\,m (dx)dt + \iint U_n\varphi_{xx}\,dx\,dt.
\end{equation}
Moreover, $g_n(x)\to g(x)$  for every $x$, and thus by bounded convergence 
$U_n(x,t)\to U(x,t)$ for every $x$ and $t$. Hence, bounded 
convergence applied to  
\eqref{approxg0} shows that $U$ satisfies \eqref{mpde1}. 
\end{proof}

\begin{proof}[Completion of the proof of Theorem~\ref{main}.] 
Let $g=o(\Psi(x))$ as $\vert x\vert\to\infty$, and let $g_n:=(g\wedge n)\vee
(-n)$ be the function $g$ truncated at $n$ and $-n$. 
Denote by $U_n(x,t)=\E g_n(X^x_t)$ the corresponding stochastic
representations. Then $U_n\to U$ pointwise as $n\to\infty$ 
by dominated convergence.
Moreover, by Lemma~\ref{Lbdd} we have
\begin{equation}
\label{approxg}
0= 2\iint U_n \varphi_t\,m (dx)dt + \iint U_n\varphi_{xx}\,dx\,dt
\end{equation}
for any $\varphi\in \mathcal D$. 
Since $|U_n(x,t)|\leq \E \vert g(X_t^x)\vert$, Lemma~\ref{loc} implies that
the functions $U_n$ are locally bounded uniformly in $n$. Thus bounded
convergence applied to  
\eqref{approxg} shows that $U$ satisfies \eqref{mpde1}. 
\end{proof}

\section{Properties of the solution}
\label{S:properties}

In this section we study monotonicity, Lipschitz continuity and convexity of the function $U$.
Let $m$ be a given speed measure and $g$ be a given continuous function with $g(x)=o(\Psi(x))$ as $\vert x\vert \to \infty$.

\begin{theorem}[Monotonicity]
\label{monotonicity}
If $g$ is non-decreasing, then also $U(x,t)$ is non-decreasing in $x$ for any fixed $t\geq 0$.
\end{theorem}

\begin{proof}
First let
$m_n$ be a sequence of measures with positive and regular densities such that $m_n$ converges to $m$ vaguely
(such a sequence can be 
constructed as convolutions $m_n:=\psi_n*(m+\frac{1}{n}\lambda)$ with a suitable sequence of
regularising kernels  $\psi_n\in\mathcal D(\R)$ in the usual way).
Let $U_n(x,t)=\E g(X^{n,x}_t)$ where $X^n$ is the generalized diffusion with speed measure $m_n$.
By property \ref{vaguelim} in Section~\ref{S:construction},
$X^{n,x}_t\to X^x_t$ almost surely as $n\to\infty$. 
As in the proof of Lemma~\ref{lowerbound}, 
the comparison result for one-dimensional diffusions 
yields that $U_n(x,t)$ is increasing in $x$. 

Assume first that $g$ is bounded. 
Then \eqref{mainu} implies, by bounded
convergence, that $U_n(x,t)\to U(x,t)$ as $n\to\infty$.
Consequently, $U(x,t)$ is increasing in $x$ and
the theorem is proved in the case of bounded $g$.

In the general case, let $g_M:=(g\wedge M)\vee(-M)$. Then the theorem follows by letting $M\to\infty$ and using
dominated convergence.
\end{proof}

\begin{theorem}[Lipschitz continuity]
Suppose that the speed measure $m$ satisfies \eqref{martingality} so that the corresponding generalized diffusion 
$X$ is a martingale. If $g$ is Lipschitz continuous with some constant $C$, i.e. $\vert g(x)-g(y)\vert\leq C\vert x-y\vert$
for all $x,y\in\R$, then so is $x\mapsto U(x,t)$ for every $t\geq 0$.
\end{theorem}

\begin{proof}
The martingale condition \eqref{martingality} implies that $\Phi (x)\to \pm \infty$ as $x\to\pm\infty$, and thus $x=o(\Psi(x))$.
If $U$ is a solution to \eqref{mpde1}--\eqref{mpde2},  
then $Cx\pm U(x,t)$ are solutions to \eqref{mpde1}
with the initial values $Cx \pm \overline g(x)$. 
Since $Cx \pm g(x)$ are non-decreasing, Theorem~\ref{monotonicity} shows that 
$Cx\pm U(x,t)$ are non-decreasing, and thus $\vert U(x,t)-U(y,t)\vert\leq C\vert x-y\vert$.
\end{proof}

\begin{theorem}[Convexity]
\label{convexity}
Assume that $m$ satisfies \eqref{martingality} so that the corresponding generalized diffusion 
$X$ is a martingale. If $g$ is convex, then the function $U(x,t)=\E g(X_t^x)$ is convex in $x$ for any fixed $t\geq 0$.
\end{theorem}

\begin{remark}
Preservation of convexity has been widely studied in the financial mathematics literature, see \cite{JT} 
and the references therein. 
Note that Theorem~\ref{convexity} includes the case of 
general martingale diffusions since we have no pointwise growth condition on the 
diffusion coefficents at infinity.
\end{remark}

\begin{proof}
Without loss of generality we assume that the measure $m$ has no point mass at zero 
(this can be achieved by translation).
First approximate $\Psi$ with smooth convex functions $\Psi_n$ such that 
\begin{itemize}
\item
$\Psi_n\geq \Psi$
\item
$\Psi_n\geq \vert x\vert^3/n$ for $\vert x\vert\geq n$
\item
$\Psi_n(x)=\Psi_n(0)+\vert x\vert^3$ for $\vert x\vert\leq 1/n$
\item
$\Psi_n\to\Psi$ pointwise as $n\to\infty$
\item 
the measure $m_n$ defined by 
$m_n(dx)=\frac{1}{2\vert  x\vert}\Psi_n^{\prime\prime}(x)\,dx$ has a
strictly positive density,
\end{itemize}
and let $X^n$ be the corresponding 
generalized diffusion. Then $m_n\to m$ vaguely as $n\to\infty$, so $X^{n,x}_t\to X^x_t$ almost surely by \ref{vaguelim} in 
Section ~\ref{S:construction}.
By \eqref{xpsi}, $\vert x\vert\leq C_1(\Psi_n(x)+1)$ holds with 
the same constant $C_1$ uniformly in $n$. Using the arguments of Lemmas~\ref{Lpsixbound} and \ref{loc}
yields that $\{g(X^{n,x}_t)\}$ is uniformly integrable in $n$ provided that $g(x)=o(\Psi(x))$ as $\vert x\vert \to\infty$.
Consequently, $U_n(x,t) :=\E g(X^{n,x}_t) \to \E g(X^x_t)=: U(x,t)$ pointwise as $n\to\infty$.

Now, if $g$ is convex, then $U_n(x,t)$ is convex in $x$ for each fixed $t$, see for example \cite{JT}.
Since the pointwise limit of a sequence of convex functions is convex, the result follows.
\end{proof}

\section{Examples}
\label{S:examples}
In this section we consider a few explicit examples of generalized diffusions and the corresponding
Feynman-Kac type theorems. 

\subsection{Brownian motion with a sticky point}
\label{examples-sticky}
In this section we study the particular case in which $m= \delta +\lambda$, where $\delta$ is a Dirac measure at 
0 and $\lambda$ is the Lebesgue measure. The corresponding generalized diffusion 
$X$ then behaves like a Brownian motion outside the point 0, which is called a sticky point for $X$ (see 
\cite{A} for a study of sticky Brownian motion).

For a given continuous initial condition $g(x)=o(\vert x\vert^3)$ we write $g=g_e+ g_o$ as the sum of an even and an odd 
function, where 
\[g_e(x):=(g(x)+g(-x))/2\]
and 
\[g_o(x):=(g(x)-g(-x))/2.\]
Now, consider the classical initial boundary value problem 
\[\left\{\begin{array}{ll}
u_t=\frac12u_{xx} & (x,t)\in(0,\infty)^2\\
u=g_e & (x,t)\in [0,\infty)\times\{0\}\\
u_t=u_x & (x,t)\in\{ 0\}\times(0,\infty).\end{array}\right.\]
By standard parabolic theory, this problem admits a unique solution in our
class, compare \cite[Section V.4]{L}. 
(Alternatively, 
for suitable $g_e$,
we can use the transformation 
$w(x,t)=u(x,t)-\int_0^t u_x(x,s)\,ds$, which solves
$w_t=\frac12w_{xx}-g_{e,x}(x)$ 
with boundary values $w(x,0)=g_e(x)$ and $w(0,t)=g_e(0)$, so $w_t(0,t)=0$;
we omit the details.)
Similarly, let $v$ be the unique solution of the initial boundary value problem
\[\left\{\begin{array}{ll}
v_t=\frac12v_{xx} & (x,t)\in(0,\infty)^2\\
v=g_o & (x,t)\in [0,\infty)\times\{0\}\\
v=0 & (x,t)\in\{ 0\}\times(0,\infty).\end{array}\right.\]
Then the function 
\[U(x,t):= u(\vert x\vert , t) + \frac{x}{\vert x\vert}v(\vert x\vert,t)\] 
solves \eqref{mpde1}--\eqref{mpde2}. 
Consequently, 
\[U(x,t)=\E g(X^x_t).\]

\begin{example}
Let $g(x)=\vert x\vert + x^2$. Then $g_e=g$ and $U(x,t)=\vert x\vert + x^2 + t$. 
\end{example}

\begin{example}
Let 
$g(x)=\vert x\vert + x^2+2\cos x - \sin |x|$. Again, $g_e=g$, and $U(x,t)=|x| +t +x^2 + e^{-t/2}(2\cos x - \sin |x|)$.
Consequently, $U(x,0)$ is $C^2$, but for all positive $t$ the solution $U$ has a kink at $x=0$.
\end{example}

\begin{remark}
Recall that solutions of parabolic equations with positive H\"older continuous diffusion coefficient gain two spatial 
derivatives. Thus starting with continuous initial data, the solution is twice 
continuously differentiable in space for any positive $t$.
The first example above shows that there is no such general gain in regularity at points where the speed measure 
is singular with respect to Lebesgue measure.
In the second example the initial data is twice continuously differentiable, but the solution is
only Lipschitz in space. Thus, in this case regularity is even lost.
\end{remark}

\subsection{Brownian motion skipping an interval}
\label{examples-skip}
Now, let the speed measure $m$ be Lebesgue measure on $\R\setminus (-1,1)$ and 0 on $(-1,1)$.
The corresponding generalized diffusion $X$ behaves like a Brownian motion outside $(-1,1)$, 
and it spends no time in $(-1,1)$.

Again, write $g=g_e+g_o$ with $g_e$ and $g_o$ as above. 
Let $u$ and $v$ be the unique solutions (of order $o(\vert x^3\vert)$) of
the problems 
\[\left\{\begin{array}{ll}
u_t=\frac12u_{xx} & (x,t)\in(1,\infty)\times(0,\infty)\\
u=g_e & (x,t)\in [1,\infty)\times\{0\}\\
u_x=0 & (x,t)\in\{ 1\}\times(0,\infty)\end{array}\right.\]
and
\[\left\{\begin{array}{ll}
v_t=\frac12v_{xx} & (x,t)\in(1,\infty)\times(0,\infty)\\
v=g_o & (x,t)\in [1,\infty)\times\{0\}\\
v_x=v & (x,t)\in\{ 1\}\times(0,\infty),\end{array}\right.\]
respectively. 
Then the function 
\[U(x,t):= \left\{\begin{array}{ll}
u(\vert x\vert , t) + \frac{x}{\vert x\vert}v(\vert x\vert,t) & \vert x\vert \geq 1\\
u(1,t)+v(1,t)x & \vert x\vert<1
\end{array}\right.\] 
solves \eqref{mpde1}--\eqref{mpde2}. 
Consequently, 
\[U(x,t)=\E g(X^x_t).\]

\begin{example} 
Consider the initial value $g(x)=\overline g(x)=\max\{ \vert x\vert, 1\}$.  
Using the recipe above, one finds that the unique solution (of order
$o(\vert x\vert^3)$) of \eqref{mpde1}--\eqref{mpde2} 
is given by 
\[U(x,t)=
\begin{cases}
|x|-2(\vert x|-1)\int_{-\infty}^{-(|x|-1)/\sqrt{t}} \phi(y)\,dy 
+2\sqrt t\phi((|x|-1)/\sqrt t), & \vert x\vert >1,
\\[3pt]
1 + \sqrt{2 t/\pi}, & \vert x\vert \leq 1  ,
\end{cases}
\]
where 
\[\phi(y):= \frac{1}{\sqrt{2\pi}}e^{-y^2/2}\]
is the density of the standard normal distribution.
It is straightforward to check that $x\mapsto U(x,t)$ is $C^1$ for $t>0$, but it fails to be $C^2$ since
\[U_{xx}(1+,t) = U_{xx}(-1-,t)=\sqrt{\frac{2}{\pi t}}\not = 0= U_{xx}(1-,t)= U_{xx}(-1+,t).\]
\end{example}

\section{Regularity}
\label{S:regularity}

\begin{theorem}
\label{regularity}
Let the assumptions in Theorem~\ref{main} hold.
Then the function $U$ solving \eqref{mpde1}--\eqref{mpde2}  
is infinitely differentiable in $t$; moreover, 
$U$ and all its derivatives with respect to $t$ and are locally
Lipschitz on $\R\times (0,\infty)$. 
\end{theorem}

In view of the examples in Section~\ref{examples-sticky}, this regularity result is sharp.
Together with the remark after Theorem \ref{main}, Theorem~\ref{regularity} yields the following.

\begin{corollary}
\label{abscont}
The distribution $U_{xx}$ is a measure on $\R\times (0,\infty)$
that is absolutely continuous with respect to $m\times \lambda$. 
Moreover, 
for each fixed $t>0$, $U_{xx}$ is a measure on $\R$
that is absolutely continuous with respect to $m$. 
\end{corollary}

\begin{remark}
The result in Corollary~\ref{abscont} can be viewed as parabolic regularity 
in our setting. 
Indeed, two spatial derivatives in regularity are gained by the solution,
when we measure regularity with respect to $m$.
\end{remark}

\begin{corollary}
\label{c1}
Assume that $m$ is absolutely continuous with respect to Lebesgue measure $\lambda$.
Then $x\mapsto \frac{\partial^k }{\partial t^k}U(x,t)$ is $C^1$ for all $t>0$ and $k\geq 0$.
Moreover, if the Radon--Nikodym derivative $\frac{dm}{d\lambda}$ is
continuous, then $x\mapsto \frac{\partial^k }{\partial t^k}U(x,t)$ is $C^2$ for all $t>0$ and $k\geq 0$.
\end{corollary}

\begin{remark}
The example in Section~\ref{examples-skip} illustrates the sharpness of
Corollary~\ref{c1}. 
\end{remark}

To prove Theorem~\ref{regularity},
let $m$ be a locally finite nonnegative Borel measure on $\R$ such that \eqref{unbounded_support}
holds, and let $g(x)=o(\Psi(x))$ as $\vert x\vert \to\infty$.
Let, as in the proof of Theorem \ref{convexity}, 
 $\Psi_n$ be a sequence of smooth convex functions satisfying
\begin{itemize}
\item
$\Psi_n\geq \Psi$
\item
$\Psi_n(x)=\Psi_n(0)+\vert x\vert^3$ for $\vert x\vert\leq 1/n$
\item
$\Psi_n\to\Psi$ pointwise as $n\to\infty$
\item 
the measure $m_n$ defined by $m_n(dx)=\frac{1}{2\vert
  x\vert}\Psi_n^{\prime\prime}(x)\,dx$ has a strictly positive density,
\end{itemize}
and let $U_n$ be the corresponding solutions
to \eqref{mpde1}--\eqref{mpde2} for the measure $m_n$. 
Lemma~\ref{Lpsixbound} then holds with the same constant $C_1$ independent
of $n$, for the corresponding generalized diffusion $X^n$,
which implies that $U_n$ is locally bounded uniformly in $n$. 
Consequently, the following lemma holds.

\begin{lemma}
For every rectangle $[a,b]\times [t_1,t_2]\subset
\R\times (0,\infty)$ we have
\begin{equation}
\label{L2bound}
\sup_n\int_a^b\int_{t_1}^{t_2} U_n^2(x,t)\,dt\,m_n(dx) <\infty.
\end{equation}
\end{lemma}

Classical regularity theory implies that each
$U_n$ is smooth on $\R\times(0,\infty)$. We have the following
$L^2$-estimates for  
derivatives of $U_n$.

\begin{lemma}
\label{dubbelintegraler}
We have
\begin{equation}
\label{estimate2}
\sup_n\int_a^b\int_{t_1}^{t_2}
\left(\frac{\partial^{k} }{\partial t^k}U_n(x,t)\right)^2\,dt\,m_n(dx) <\infty 
\end{equation}
and
\begin{equation}
\label{estimate1}
\sup_n\int_a^b\int_{t_1}^{t_2} \left(\frac{\partial^{k+1} }{\partial x\partial t^k}U_n(x,t)\right)^2\,dt\,dx <\infty
\end{equation}
for all $k\geq 0$ and any rectangle $[a,b]\times [t_1,t_2]\subset
\R\times (0,\infty)$.
\end{lemma}

\begin{proof}
We assume, somewhat more generally,
that $U_n$ is any sequence of solutions to \eqref{mpde1}, for the
measure $m_n$, such that \eqref{L2bound} holds for any rectangle.

To simplify the notation, we suppress the dependence on $n$ and
write $u=U_n$ and $\mu=m_n$, respectively.
By construction, $m_n$ has a smooth density, which we also denote by $\mu$.
If $\psi\in C^2_c(\R)$, then using \eqref{mpde1} and integrations by parts we
find that  
\begin{eqnarray}
\label{ip}
\frac{d}{d t}\int_\R \psi u^2\,\mu(dx) &=& \int_\R 2\psi u u_t\,\mu(dx)
= \int_\R \psi uu_{xx} \,dx\\
\notag
&=& -\int_\R \psi  u_x^2\,dx-\int_\R\psi_x u u_x\,dx\\
\notag
&=&  -\int_\R \psi  u_x^2\,dx + \frac{1}{2}\int_\R\psi_{xx} u^2\,dx,
\end{eqnarray}
and thus
\begin{equation}
\label{eq1}
\int_\R\psi u_x^2\,dx = \frac{1}{2}\int_\R\psi_{xx}u^2\,dx -\frac{d}{dt}\int_\R\psi u^2\,\mu(dx).
\end{equation}
If $0<\tO<t_2$, 
integrating \eqref{eq1} for $t\in (\tO,t_2)$ gives
\begin{eqnarray}
\label{eq2}
\int_{\tO}^{t_2}\int_\R \psi u_x^2\,dx\,dt &=& \frac{1}{2}\int_{\tO}^{t_2}\int_\R\psi_{xx}u^2\,dx\,dt +\int_\R\psi(x) u^2(x,\tO)\,\mu(dx) \\
\notag
&&- \int_\R\psi(x) u^2(x,t_2)\,\mu(dx).
\end{eqnarray}
Assume further that $\psi\geq 0$ and integrate again for $\tO\in(\ti/2,\ti)$ with $t_1<t_2$ to obtain
\begin{equation}
\label{eq3}
\frac{\ti}{2}\int_{\ti}^{t_2}\int_\R \psi u_x^2 \,dx\,dt \leq
\frac{\ti}{4}\int_{\ti/2}^{t_2}\int_\R \vert \psi_{xx}\vert u^2 \,dx\,dt
 + \int_{\ti/2}^{\ti}\int_\R \psi u^2 \,\mu(dx)\,dt
\end{equation}
Let $a<b$. We choose $\psi=\psi_n$ depending on $n$ as follows. 
Find $a_1<a_2<a-1$ and $a_4>a_3>b+1$ with $a_1,a_2,a_3,a_4\in \supp m$
and $a_2-a_1>3$, $a_4-a_3>3$. Denote by $I_1,I_2,I_3,I_4$ the disjoint open
intervals  
in $\R\setminus [a,b]$ such that $I_j=(a_j-1,a_j+1)$. Then $m(I_j)>0$ for $j=1,2,3,4$, 
and for $n$ large $\mu(I_j)>\frac{1}{2} m(I_j)$.
Define $\psi$ by
\begin{equation}
\label{psi}
\psi_{xx}=\left\{\begin{array}{cl}
\mu/\mu(I_1) & \mbox{in } I_1\\
-\mu/\mu(I_2) & \mbox{in } I_2\\
-\dc\mu/\mu(I_3) & \mbox{in } I_3\\
\dc\mu/\mu(I_4) & \mbox{in } I_4\\
0 & \mbox{elsewhere}\end{array}\right.
\end{equation}
with $\psi(x)=0$ for $x\leq a_1-1$ and $\dc>0$ to be chosen.
Then 
\begin{equation*}
\psi_x=\left\{\begin{array}{cl}
0 & \mbox{on } (-\infty,a_1-1)\\
1 & \mbox{on }(a_1+1,a_2-1)\\
0 & \mbox{on }(a_2+1,a_3-1)\\
-\dc  & \mbox{on }(a_3+1,a_4-1)\\
0 & \mbox{on }(a_4+1,\infty)\end{array}\right.
\end{equation*}
Hence $\psi$ is constant and larger than $a_2-1-(a_1+1)\geq 1$ on $[a,b]\subset (a_2+1,a_3-1)$, smaller than
$a_2+1-(a_1-1)=a_2-a_1+2$ everywhere,
and by choosing a suitable $\dc$ we have $\psi=0$ on $(a_4+1,\infty)$. Then $\psi\in C^1_c(\R)$,
and $\psi_{xx}$ is bounded and 
continuous everywhere but at a finite number of points. By an approximation
argument (or extending \eqref{ip}),  
the inequality \eqref{eq3} holds also for such functions. Moreover, $\dc$ is
uniformly bounded for $n$ large. 
Thus, for large $n$, $\vert\psi_{xx}\vert\leq  C \mu 1_{[a_0,a_5]}$ and $\psi\leq C1_{[a_0,a_5]}$ with $a_0:=a_1-1$, $a_5:=a_4+1$
and $C$ independent of $n$. Thus \eqref{eq3} implies,
by our assumption \eqref{L2bound},
\[\frac{\ti}{2}\int_{\ti}^{t_2}\int_a^b u_x^2 \,dx\,dt \leq
C\int_{\ti/2}^{t_2}\int_{a_0}^{a_5}  u^2 \mu(dx)\,dt
 +C \int_{\ti/2}^{\ti}\int_{a_0}^{a_5}  u^2 \,\mu(dx)\,dt\leq C,\]
i.e. for any $[a,b]\times [\ti,t_2]\subset\R\times(0,\infty)$,
\begin{equation}
\label{c}
\int_{\ti}^{t_2}\int_a^b u_x^2\,dx\,dt\leq C
\end{equation}
with $C$ independent of $n$ (but depending on $a,b,\ti,t_2$). 
This is \eqref{estimate1} for $k=0$.

Again, let $\psi\in C_c^2(\R)$ with $\psi\geq 0$. Then 
\begin{eqnarray*}
\frac{d}{d t}\int_\R \psi u_x^2\,dx &=& 2\int_\R \psi u_x u_{xt}\,dx = -2\int_\R(\psi u_x)_x u_t\,dx\\
&=& -2\int_\R \psi u_{xx} u_t\,dx -2\int_\R \psi_x u_xu_t\,dx\\
&=& -4\int_\R \psi u^2_t\,\mu(dx) - \int_\R\frac{\psi_xu_x u_{xx}}{\mu}\,dx\\
&=&  -4\int_\R \psi u^2_t\,\mu(dx) +\frac{1}{2} \int_\R \frpsimu_x u_x^2\,dx
\end{eqnarray*}
or 
\begin{equation}
\label{eq5}
4\int_\R \psi u^2_t\,\mu(dx) = -\frac{d}{d t}\int_\R \psi u_x^2\,dx 
+\frac{1}{2} \int_\R \frpsimu_x u_x^2\,dx.
\end{equation}
Integrate for $t\in (\tO,t_2)$ to obtain
\begin{eqnarray}
\label{eq5integrated}
4\int_{\tO}^{t_2}\int_\R \psi u^2_t\,\mu(dx)\,dt &=& \int_\R \psi u_x^2(x,\tO)\,dx
 -\int_\R \psi u_x^2(x,t_2)\,dx\\
\notag
&&
+\frac{1}{2} \int_{\tO}^{t_2}\int_\R \frpsimu_x u_x^2\,dx\,dt.
\end{eqnarray}
Integrating once more for $\tO\in (\ti/2,\ti)$ yields
\begin{equation}
\label{eq6}
2\ti\int_{\ti}^{t_2}\int_\R \psi u^2_t\,\mu(dx)\,dt  \leq
\int_{\ti/2}^{\ti}\int_\R \psi u_x^2\,dx\,dt +\frac{\ti}{4}
\int_{\ti/2}^{t_2}\int_\R
\left\vert \frpsimu_x\right\vert u_x^2\,dx\,dt.
\end{equation}
We choose again $\psi=\psi_n$ depending on $n$. This time, with $a<b$ given and notations as above,
let $\varphi_1\in C^\infty_c(I_1)$ and $\varphi_4\in C^\infty_c(I_4)$ with $\varphi_1,\varphi_4\geq 0$ and 
$\varphi_1(a_1),\varphi_4(a_4)>0$. Then $\int \varphi_j\,\mu(dx)\to
\int\varphi_j \,m(dx)>0$ for $j=1,4$ 
as $n\to\infty$.
Define $\psi=\psi_n\in C_c^\infty(\R)$ by 
\begin{equation}
\label{psi2}
\psi_x=\left\{\begin{array}{cl}
\frac{\varphi_1 \mu}{\int\varphi_1\mu(dx)} & \mbox{on }I_1\\
-\frac{\varphi_4 \mu}{\int\varphi_4\mu(dx)} & \mbox{on }I_4\\
0 & \mbox{elsewhere}.\end{array}\right.
\end{equation}
Then $\psi=1$ on $[a_1+1,a_4-1]\supset [a,b]$. Moreover, $0\leq \psi\leq 1_{[a_0,a_5]}$ and $\vert (\frac{\psi_x}{\mu})_x\vert\leq C1_{[a_0,a_5]}$
(for $n$ large) so \eqref{eq6} implies
\[2\ti\int_{\ti}^{t_2}\int_a^b u^2_t\,\mu(dx)\,dt  \leq \int_{\ti/2}^{\ti}\int_{a_0}^{a_5} u_x^2\,dx\,dt +C \int_{\ti/2}^{t_2}\int_{a_0}^{a_5} u_x^2\,dx\,dt.\]
Consequently, \eqref{c} yields, for any $[a,b]\times[\ti,t_2]\subset \R\times (0,\infty)$,
\begin{equation}
\label{eq7}
\int_{\ti}^{t_2}\int_a^b u^2_t\,\mu(dx)\,dt\leq C
\end{equation}
with $C$ independent of $n$.

We have used \eqref{L2bound} to show that \eqref{c} and \eqref{eq7} hold.
Since $u_t=\frac{\partial U_n}{\partial t}$ is another solution of the equation (with $m=m_n$),
the estimates \eqref{estimate2} and \eqref{estimate1} hold by induction
for all $k\ge0$. 
\end{proof}

\begin{lemma}
\label{enkelintegraler}
For each fixed integer $k\geq 0$,
\begin{equation}
\label{enkelintegral1}
\sup_n\int_a^b\left(\frac{\partial^{k} }{\partial t^k}U_n(x,t)\right)^2\,m_n(dx) <\infty
\end{equation}
and
\begin{equation}
\label{enkelintegral2}
\sup_n\int_a^b \left(\frac{\partial^{k+1} }{\partial x\partial t^k}U_n(x,t)\right)^2\,dx <\infty
\end{equation}
uniformly in $t\in[t_3,t_4]$, for every rectangle $[a,b]\times [t_3,t_4]\subset
\R\times (0,\infty)$.
\end{lemma}

\begin{proof}
For simplicity of notation we consider the case $k=0$ (the general case being completely 
analogous).
Fix $a<b$, let $\psi$ be defined by \eqref{psi} and with compact support, and fix $[t_3,t_4]\subset (0,\infty)$. 
For any $t_1,t_2\in [t_3,t_4]$, \eqref{eq2} together with \eqref{estimate2} and \eqref{estimate1} yield
\begin{align*}
& \left\vert\int_\R\psi (x)u^2(x,t_1)\,\mu(dx)- \int_\R\psi
  (x)u^2(x,t_2)\,\mu(dx)\right\vert \\ 
&\qquad
 \leq\int_{t_3}^{t_4}\int_\R \psi u_x^2\,dx\,dt + \frac{1}{2}
 \int_{t_3}^{t_4} \int_a^b \vert \psi_{xx}\vert u^2\,dx\,dt \\
&\qquad
 \leq C\int_{t_3}^{t_4}\int_{a_0}^{a_5} u_x^2\, dx\,dt + C\int_{t_3}^{t_4}\int_{a_0}^{a_5} u^2\,\mu(dx)\,dt\leq C.
\end{align*}
Since further, by \eqref{estimate2} again, 
\[\int_{t_3}^{t_4}\int_\R \psi u^2\,\mu(dx)\,dt\leq C \int_{t_3}^{t_4}\int_{a_0}^{a_5} u^2\,\mu(dx)\,dt\leq C,\]
it follows that, uniformly for $t\in[t_3,t_4]$,
\[\int_\R \psi u^2\,\mu(dx)\leq C\]
and thus, for any $[a,b]$ and $[t_3,t_4]\subset (0,\infty)$ and every $t\in [t_3,t_4]$
\begin{equation}
\label{eq8}
\int_a^b u^2(x,t)\mu(dx) \leq C,
\end{equation}
which is \eqref{enkelintegral1} for $k=0$.

Similarly, inserting $\psi$ defined as in \eqref{psi2} and with compact support in \eqref{eq5integrated} gives,
for any $t_1,t_2\in[t_3,t_4]$,
\begin{align*}
& \left\vert \int_\R \psi(x) u_x^2(x,t_1)\,dx - \int_\R \psi(x) u_x^2(x,t_2)\,dx\right\vert  \\
&\qquad 
\leq4
\int_{t_3}^{t_4}\int_\R \psi u_t^2\,\mu(dx)\,dt + \int_{t_3}^{t_4}\int_\R \left\vert \frpsimu_x\right\vert u_x^2\,dx\,dt\\
&\qquad
\leq C\int_{t_3}^{t_4}\int_{a_0}^{a_5}  u_t^2\,\mu(dx)\,dt + \int_{t_3}^{t_4}\int_{a_0}^{a_5} u_x^2\,dx\,dt\leq C.
\end{align*}
Since further,  by \eqref{estimate1}, 
\[\int_{t_3}^{t_4}\int_\R \psi  u_x^2\,dx\,dt\leq \int_{t_3}^{t_4}\int_{a_0}^{a_5}u_x^2\,dx\,dt\leq C,\]
it follows that for every $t\in[t_3,t_4]$,
\[\int_\R \psi u_x^2(x,t)\,dx \leq C.\]
Consequently, 
\begin{equation}
\label{eq9}
\int_a^b  u_x^2(x,t)\,dx \leq C
\end{equation}
uniformly in $n$ and $t\in[t_3,t_4]$, for every $[a,b]\times[t_3,t_4]\subset \R\times (0,\infty)$, so 
\eqref{enkelintegral2} holds for $k=0$.
\end{proof}

\begin{lemma}
\label{rehnskiold}
The functions $\frac{\partial^{k}U_n }{\partial t^k}$ and $\frac{\partial^{k+1} U_n}{\partial x\partial t^k}$ with $k\geq 0$
are locally bounded on $\R\times (0,\infty)$, uniformly in $n$. Thus $U_n$ and all its time derivatives are locally Lipschitz 
on $\R\times (0,\infty)$, uniformly in $n$.
\end{lemma}

\begin{proof}
First note that it follows from \eqref{enkelintegral2} that $u:=U_n$ is
locally H\"older(1/2)-continuous in the spatial variable, uniformly in $n$. Indeed,
to see this, note that by the Cauchy--Schwarz inequality
\begin{equation}
\label{stenbock}
\vert u(y,t)-u(x,t)\vert = \left\vert \int_x^y u_z(z,t)\,dz\right\vert
\leq C\sqrt{\vert y-x\vert}
\end{equation}
uniformly in $x,y\in[a,b]$ and $t\in [t_3,t_4]$, for every rectangle $[a,b]\times [t_3,t_4]\subset
\R\times (0,\infty)$.

Now, let $I=[a,b]\subset\R$ be a given non-empty interval. 
By \eqref{unbounded_support}, we may 
increase $I$ so that $\eta:=m(I^\circ)>0$. 
It follows from the vague convergence of $\mu=m_n$ to $m$ that $\mu(I)\geq \eta/2$ for 
sufficiently large $n$. Pick $x_0 \in I$ and $t_3,t_4\in (0,\infty)$, and note that for $t\in (t_3,t_4)$
the local H\"older continuity of $u$ implies that
\[ \vert u(x,t)-u(x_0,t)\vert \leq  C\sqrt{b-a}\]
for all $x\in I$. It therefore follows from \eqref{enkelintegral1} that 
\begin{equation*}
\vert u(x_0,t)\vert\leq  C\sqrt{b-a}+ \sqrt{2C/\eta}
\end{equation*}
for $n$ sufficiently large. Consequently, $u$ is locally bounded uniformly in $n$.
The case of time derivatives of $u$ is completely analogous.

Similarly, the Cauchy--Schwarz inequality applied to $u_x=\frac{\partial U_n}{\partial x}$ yields
\begin{eqnarray}
\label{mazepa}
\vert u_x(y,t)-u_x(x,t)\vert &=& \left\vert \int_x^y u_{zz}(z,t)\,dz\right\vert 
= \left\vert \int_x^y 2u_{t}(z,t)\mu(dz)\right\vert\\
\notag
&\leq& 2\left(\int_x^y u_t^2(z,t)\,\mu(dz)\right)^{1/2}
\left(\mu(x,y)\right)^{1/2}\\
\notag&\leq& C \left(\mu(a,b)\right)^{1/2}\leq C,
\end{eqnarray}
where the last inequality holds since $\limsup_{n\to\infty}\mu(a,b)\leq m[a,b]<\infty$.
Note that \eqref{mazepa} holds
uniformly in $x,y\in[a,b]$ and $t\in [t_3,t_4]$, for every rectangle $[a,b]\times [t_3,t_4]\subset
\R\times (0,\infty)$. 
Together with \eqref{enkelintegral2} for $k=0$, this gives the desired local bound of $u_x=\frac{\partial U_n}{\partial x}$.
The case of time derivatives of $u_x$ is completely analogous.
\end{proof}

\begin{proof}[Proof of Theorem~\ref{regularity}.]
Since $m_n\to m$ vaguely as $n\to\infty$, $X^{n,x}_t\to X^x_t$ almost surely by \ref{vaguelim} in 
Section ~\ref{S:construction}. 
Arguing as in the proof of Theorem~\ref{convexity},
$U_n(x,t) =\E g(X^{n,x}_t) \to \E g(X^x_t)=U(x,t)$ pointwise as $n\to\infty$.
By Lemma~\ref{rehnskiold} and the Arzela--Ascoli theorem,
$U$ and its partial derivatives with respect to time exist and they are
locally Lipschitz continuous.
\end{proof}

\end{document}